\documentclass[reqno]{amsart}
\usepackage[utf8]{inputenc}
\usepackage[affil-it]{authblk}
\usepackage{amsthm,amsmath,amssymb,amsfonts,amsaddr}
\usepackage{graphicx,url,bbm,xspace}
\usepackage{cases}

\usepackage[bookmarks, colorlinks=true, pdfstartview=FitV, linkcolor=blue, citecolor=blue, urlcolor=blue]{hyperref}

\newtheorem{theorem}{Theorem}
\newtheorem{lemma}[theorem]{Lemma}
\newtheorem{corollary}[theorem]{Corollary}
\theoremstyle{definition}
\newtheorem{definition}[theorem]{Definition}
\newtheorem{example}[theorem]{Example}
\newtheorem{remark}[theorem]{Remark}

\numberwithin{theorem}{section}

\newcommand*{\addheight}[2][.5ex]{%
  \raisebox{0pt}[\dimexpr\height+(#1)\relax]{#2}%
}
\newcommand{\argmin}{\operatorname{argmin}}

\newcommand{\bR}{\ensuremath{\mathbbm{R}}\xspace}
\newcommand{\bQ}{\ensuremath{\mathbbm{Q}}\xspace}
\newcommand{\bZ}{\ensuremath{\mathbbm{Z}}\xspace}
\newcommand{\bN}{\ensuremath{\mathbbm{N}}\xspace}
\newcommand{\ee}{\operatorname{e}}
\newcommand{\eps}{\varepsilon}
\newcommand{\cosn}{\cos(n)}

\title[$\mu(\pi)$ and powers of ${\protect\cosn}$.]
{The irrationality measure of $\pi$ as seen through the eyes of $\cos(n)$}
\author{Sully F. Chen}
\email{sullyche@usc.edu}
\address{California Polytechnic University, San Luis Obispo \\ University of Southern California}

\author{Erin P. J. Pearse}
\email{epearse@calpoly.edu}
\address{California Polytechnic University, San Luis Obispo}

\date{\today}

\begin{document}

\maketitle

\begin{abstract}
For different values of $\gamma \geq 0$, analysis of the end behavior of the sequence $a_n = \cos (n)^{n^\gamma}$ yields a strong connection to the irrationality measure of $\pi$. We show that if $\limsup |\cos n|^{n^2} \neq 1$, then the irrationality measure of $\pi$ is exactly 2. We also give some numerical evidence to support the conjecture that $\mu(\pi)=2$, based on the appearance of some startling subsequences of $\cos(n)^n$.
\end{abstract}

\section{Introduction}
\subsection{Irrationality Measure}
An irrational number $\alpha$ is considered to be well-approximated by rational numbers iff for any $\nu \in \bN$, there are infinitely many choices of $\frac pq$ that satisfy \eqref{eqn:approximability-inequality}. 
%For any irrational number $\alpha$, we can measure how well $\alpha$ can be approximated by rational numbers by studying the inequality 
\begin{align}\label{eqn:approximability-inequality}
    \left | \alpha - \frac{p}{q}\right | < \frac{1}{q^\mu}, 
    \qquad \text{for some } \nu \in \bN.
\end{align}
The idea is that if the set
\begin{equation}
    U(\alpha,\nu) = \left\{\frac{p}{q} \in \bQ : 0 < \left | \alpha - \frac{p}{q}\right | < \frac{1}{q^\nu}\right\}
    \label{eqn:U-set}
\end{equation}
is infinite for a fixed $\nu$, then $U(\alpha,\nu)$ contains rationals with arbirarily large denominators and so one can find a sequence of rationals $\frac {p_n}{q_n} \in U(\alpha,\nu)$ which converge to $\alpha$. One could say that this sequence converges to $\alpha$ ``at rate $\nu$''. When this can be accomplished for any $\nu \in \bN$, the number $\alpha$ can be approximated by sequences with arbitrarily high rates of convergence. Such numbers are called \emph{Liouville numbers} after Joseph Liouville who proved that all such numbers are transcendental, thereby giving the first proof of the existence of transcendental numbers \cite{Liouville}. See Appendix A for a proof of the transcendentality of Liouville numbers.

When the number $\alpha$ is not well-approximated by a sequence of rationals, one can grade ``how approximable'' $\alpha$ is in terms of the optimal exponent for which an approximating sequence of rationals may be found, and this leads to the notion of irrationality measure (also variously called the Liouville-Roth irrationality measure, irrationality exponent, approximation exponent, or Liouville-Roth constant); cf. \cite{Bugeaud, Weisstein3}.

\theoremstyle{definition}
\begin{definition}[Irrationality Measure]\label{def:irrationality-measure}
% For $\alpha \in \bR$, define the sets
% \begin{equation*}
%     U(\alpha,\nu) = \left\{\frac{p}{q} \in \bQ : 0 < \left | \alpha - \frac{p}{q}\right | < \frac{1}{q^\nu}\right\}
% \end{equation*}
%Let $X$ be the set of all positive real numbers $\mu$ such that finitely many rationals $\frac{p}{q}$ satisfy \eqref{eq:1}. Then 
The \emph{irrationality measure} of $\alpha \in \bR \setminus \bQ$ is  %$\mu (\alpha) = \inf X$.
\begin{equation*}% \label{eq:irrationality-measure}
    \mu (\alpha) = \inf \{\nu : |U(\alpha,\nu)| < \infty\}.
\end{equation*}
\end{definition}

%The irrationality measure of a number describes how ``well'' an irrational number can be approximated by rational numbers, as can be seen in the following example. 

% For any value $\nu < \mu(\alpha)$, we consider the infinite set
% \begin{equation*}
%     U(\alpha,\nu) = \left\{\frac{p}{q} \in \bQ : 0 < \left | \alpha - \frac{p}{q}\right | < \frac{1}{q^\nu}\right\}
% \end{equation*}
% of all rationals satisfying \eqref{eqn:approximability-inequality} as an approximation of $\alpha$. 
This means, for example, that except possibly for at most finitely many ``lucky'' choices of $\frac pq$, every rational approximation $\frac pq$ that gives $n+1$ correct decimal digits of $\alpha$ must satisfy
\[\frac {1}{10^{n}} \geq \left|\alpha-{\frac {p}{q}}\right|\geq {\frac {1}{q^{\mu(\alpha) +\varepsilon }}}\] 
for any $\varepsilon>0$.

\begin{example}
Liouville numbers are precisely those numbers $\alpha$ with $\mu(\alpha)=\infty$; cf. \cite{Bugeaud}. The most well-known example is Liouville's Constant \cite{Weisstein1}, defined as:
\[\alpha = \sum_{k=1}^\infty 10^{-k!} = 0.110001... \; .\] 
One can see by inspection that the partial sums in the $10$-ary expansion of $\alpha$ (more commonly called the \emph{decimal expansion}) of $\alpha$ approach the true value extremely quickly. If $s_n$ denotes the $n^{\textrm{th}}$ partial sum, we have the first few rational approximations and their corresponding error bounds \\
\[\begin{tabular}{|r|l|l|l|l|} \hline
    $n$ & 1 & 2 & 3 & 4  \\ \hline
    $s_n$ & 0.1 & $0.11$ & $0.110001$ & $0.11000100...001$  \\
    $|\alpha-s_n| \leq $ & $2 \times 10^{-2}$ & $2 \times 10^{-6}$ & $2 \times 10^{-24}$ & $2 \times 10^{-120}$   \\ \hline
\end{tabular}\]
%Notice how $\alpha$ is quite "close" to the rational number $1.1$, and even closer to $1.1001$, and so forth. In a sense, the error in approximation of $\alpha$ by rational numbers decreases extremely quickly as more numbers past the decimal are added. 
Observe that the number of decimal places for which $s_n$ agrees with $\alpha$ increases much faster than $n$. In terms of Definition~\ref{def:irrationality-measure}, we have found that $U(\alpha,2)$ contains infinitely many points. %In fact, it can be shown that $U(\alpha,\mu)$ contains infinitely many points for any $\mu \in \mathbb{R}$, a very special quality of the Liouville Constant, and more generally, Liouville numbers \cite{Weisstein2} 

In contrast, for $\sqrt2$ (whose irrationality measure is equal to 2), no such patterns exist in the $10$-ary expansion, or indeed, in the $q$-ary expansion of $\sqrt2$ for any $q \in \mathbb{Q}^+$. 
\end{example}

It is clear that there exists a deep underlying meaning and complexity to the irrationality measure of a number, and thus any efforts towards determining the exact values of the irrationality measures of different numbers would yield considerable insight.

It can be shown that $\mu(\alpha) =1$ for any rational number $\alpha \in \mathbb{Q}$,  and that $\mu(\alpha) =2$, for every  algebraic number $\alpha$ of order greater than 1, from Roth's theorem on arithmetic progressions; cf. \cite{Croot}.\footnote{This is a deep result, and Roth's work in this area earned him the 1958 Fields Medal.} However, transcendental numbers may have any irrationality measure greater than or equal to 2. There is no known method to compute or determine the exact value of the irrationality measure of any transcendental number, or even estimate the value. Upper bounds have been proven for many interesting transcendental constants, for example, $\mu (\pi) \leq 7.6063$ \cite{V. Kh. Salikhov1} and $\mu (\ln 3) \leq 5.125$ \cite{V. Kh. Salikhov2}.  Proofs of exact values are rare, and one such gem is $\ee$, for which it is known that $\mu(\ee)=2$ \cite{Bugeaud}. This is perhaps not surprising in light of the fact that almost all numbers have irrationality measure 2.

An exact value for the irrationality measure of the most famous of mathematical constants, $\pi$, has evaded us thus far, though it is widely believed that $\mu(\pi)=2$. Using Corollary~\ref{thm:corollary-to-main}, we provide  further numerical evidence that $\mu(\pi)=2$ in \S\ref{sec:numerics}.

Our main result concerns the following set:
\begin{equation} \label{eq:Phi}
    \Phi = \{\gamma \geq 0 : \limsup \{|\cos k|^{k^\gamma}\}_{k \in \bN} = 1\}.
\end{equation}

\begin{theorem}\label{thm:main-result}
   For $\Phi$ as defined in \eqref{eq:Phi}, we have $\sup \Phi = 2\mu(\pi) - 2$.
\end{theorem}

\begin{corollary}\label{thm:corollary-to-main}
    If $\sup \Phi = 2$, then $\mu(\pi)=2$.
\end{corollary}

\begin{remark}[The relation of $\mu (\pi)$ to the end behavior of special sequences]
The value of $\mu (\pi)$ is closely tied to the behavior and convergence of special sequences of transcendental numbers. For example, Alekseyev proved that the convergence of the Flint Hills series directly implies $\mu (\pi) < 2.5$, in \cite{Max A. Alekseyev}. 

The investigation in the present paper began with a question raised by a student regarding the behavior of the sequence $a_n = \cos(n)^n$. Since $|\cos n|<1$, one might expect that raising this number to a large power would cause $a_n \to 0$. However, preliminary numerical investigations revealed that this is not the case. Indeed, while most values of this sequence are extremely close to 0, certain subsequences form oscillations with amplitude 1. The top of Figure~\ref{fig:1000000} shows the first million terms of the sequence, and the bottom shows a magnification of just the first 100,000 terms, which reveals several rapidly decaying subsequences. In this note, we examine the sequence $a_n = \cos(n)^{n^\gamma}$, for $\gamma \geq 0$, and show that is closely related to the value of $\mu (\pi)$.
\end{remark}

\begin{figure}
    \addheight{
    \begin{tabular}{cc}
    \includegraphics[width=120mm]{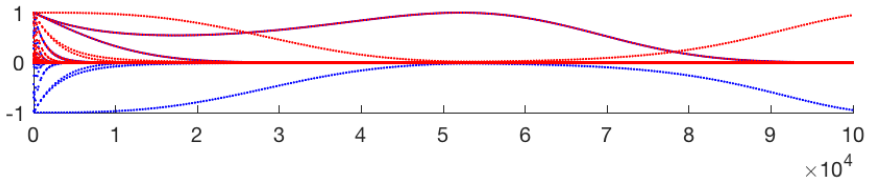}
    \\ 
    \includegraphics[width=120mm]{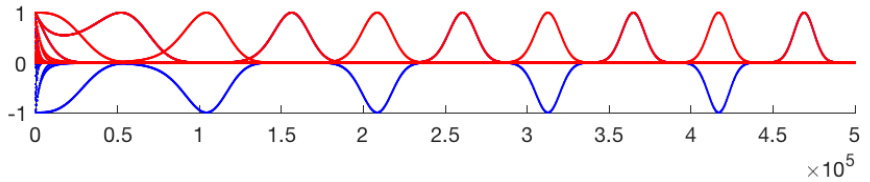}
    \end{tabular}}
      
    \caption{\small The top plot shows the first 20,000 terms of $a_n = \cos(n)^n$; the bottom figure shows the first 1,000,000. The subsequence $a_{2k}$ is plotted in red and the subsequence $a_{2k+1}$ is plotted in blue. 
    %Values with $|a_n| < 0.001$ have been deleted for clarity.
    Note: each peak actually corresponds to a different subsequence, an arithmetic progression of the form $n_k = 355k + 22j$; see \S\ref{sec:subsequences} for details.
    }
    \label{fig:1000000}
\end{figure}

\section{Proof of the main result}
%{The Relationship between $\mu(\pi)$ and the End Behavior of $\left ( |\cos k|^{k^\gamma} \right )_{k = 0}^\infty$}

\begin{definition}\label{def:P_and_Q}
We define a function $Q(n): \bZ \mapsto \bZ$ such that for any integer $n$, $Q(n)$ maps to the integer minimizing $|n - Q(n)\pi|$. In other words, if $[x]$ denotes the nearest integer to $x$,\footnote{Note that $[x]$ is \emph{not} the same as the floor or ceiling function.} then
%\begin{equation} \label{eq:P}
%    P(n) = [n \pi] = \argmin \{|P(m)-n \pi| : m \in \bZ\}
%\end{equation}
%and
\begin{equation} \label{eq:Q}
    Q(n) = \left[\frac n\pi\right] = \argmin \{|n-m \pi| : m \in \bZ\}.
\end{equation}
\end{definition}

\subsection*{Outline of the proof of Theorem~\ref{thm:main-result}} 
We prove Theorem~\ref{thm:main-result} by establishing
\begin{subnumcases}{\limsup \{|\cos k|^{k^\gamma}\}_{k \in \bN} =}
    1, & $\gamma < 2\mu(\pi) - 2$, \label{eqn:main-result-g-case-i} \\
    0, & $\gamma > 2\mu(\pi) - 2$. \label{eqn:main-result-g-case-ii}
\end{subnumcases}
The proof will require a couple of technical lemmas, which we state and prove before proceeding. % to the main proof of Theorem~\ref{thm:main-result}.
From Lemma~\ref{thm:p/Q(p)->pi}, it follows that $Q(n)=\mathcal{O}(n)$, $n \to \infty$, and this is observation used to show \eqref{eqn:main-result-g-case-i}. 

For \eqref{eqn:main-result-g-case-ii}, we establish a bound of the form 
\[|\cos p|^{p^\gamma} \leq \left( 1 - \frac{1}{\phi_0(p)}\right )^{\phi_1(p)}\]
for 
\[\phi_0(p) = \frac{24Q(p)^{4(\mu(\pi) - 1 + \lambda)}}{12Q(p)^{2(\mu(\pi) - 1 + \lambda)} - 1}
\qquad\text{and} \qquad \phi_1(p) = p^\gamma,
\]
and then apply Lemma~\ref{thm:phi-identity}.

\begin{lemma}\label{thm:p/Q(p)->pi}
%$P(n) = \mathcal{O}(n)$, as $n \to \infty$, and $Q(n) = \mathcal{O}(n)$, as $n \to \infty$.
$\lim_{n \to \infty} \frac{n}{Q(n)} = \pi$.
% $Q(n) = \mathcal{O}(n)$, as $n \to \infty$.
\end{lemma}
\begin{proof}
%To show that $P(n)$ and $Q(n)$ grow linearly, we must show:
%$$\lim_{n \to \infty} \frac{P(n)}{n} = c_1 \in \bR \quad \mathrm{and} \lim_{n \to \infty} \frac{Q(n)}{n} = c_2 \in \bR$$
This is immediate from Definition~\ref{def:P_and_Q} by Diophantine approximation.
% \begin{align*}
% %\lim_{n \to \infty} \frac{P(n)}{n} = \pi, 
% %&\qquad \textrm{and}\qquad
% \lim_{n \to \infty} \frac{n}{Q(n)} &= \pi. 
% \qedhere
% \end{align*}
% %and we are done.
\end{proof}

\begin{lemma}\label{thm:phi-identity}
Let $\phi_0(x)$ and $\phi_1(x)$ be \bR-valued functions satisfying 
\begin{align*}
    \lim_{x \to \infty} \phi_0(x) = \infty
    \qquad\text{and}\qquad
    \lim_{x \to \infty} \phi_1(x) = \infty,
    %\quad \text{as } x \to \infty
\end{align*} 
If $\phi_0(x) \neq 0$ for all $x \in \bR$ and $\phi_0 = o(\phi_1)$, as $x \to \infty$, then
\begin{align}\label{eqn:phi-identity}
    \lim_{x \to \infty}\left( 1 - \frac{1}{\phi_0(x)}\right )^{\phi_1(x)} = 0.
\end{align}
\end{lemma}
\begin{proof}
%Observe that for all $x \in \bR$:
%$$
%\left ( 1 + \frac{1}{x}\right )^x < e
%$$
%We have:
%$$
%(1 + x)^r = \left ( 1 + \frac{rx}{r} \right )^r
%$$
%Let $a$ = $\frac{r}{rx}$:
%$$
%\left ( 1 + \frac{rx}{r} \right )^r = \left ( 1 + \frac{1}{a} \right )^{arx} = \left [ \left ( 1 + \frac{1}{a} \right )^{a} \right ]^{rx} < e^{rx}
%$$
%So:
Observe that $r(r-1)(r-2)\dots(r-n+1) \leq r^n$, so comparing the Taylor series for $(1+x)^r$ and $\ee^{rx}$ gives
\[(1 + x)^r \leq \ee^{rx}.\]
Now making the substitutions $x = -\frac{1}{\phi_0(x)}$ and $r = \phi_1(x)$, we obtain
\[\left ( 1 - \frac{1}{\phi_0(x)}\right )^{\phi_1(x)} \leq \exp \left (-\frac{\phi_1(x)}{\phi_0(x)} \right ),
\]
and the right side tends to 0 as $x \to \infty$ because $\phi_0 = o(\phi_1)$.
%\[
%0 \leq \lim_{x \to \infty} \left ( 1 - \frac{1}{\phi_0(x)}\right )^{\phi_1(x)} \leq %\lim_{x \to \infty} \exp \left (-\frac{\phi_1(x)}{\phi_0(x)} \right ) = 0
%\]
\end{proof}

\begin{proof}[Proof of Theorem~\ref{thm:main-result}]
\underline{Step 1:} we prove \eqref{eqn:main-result-g-case-i}.  
% that  
% \begin{equation}\label{eqn:1_for_g<2mu}
%     \limsup \{|\cos k|^{k^\gamma}\}_{k \in \bN} = 1,
% \qquad \textrm{for all } \gamma < 2\mu(\pi) - 2.
% \end{equation}
%
By Definition~\ref{def:irrationality-measure}, %(irrationality measure), 
there are infinitely many integers $p,q > 0$ such that:
\begin{equation*}
0 < \left | \pi - \frac{p}{q}\right | < \frac{1}{q^{\mu(\pi) - \eps}},  \quad \eps > 0.
\end{equation*}
We will make repeated use of the following Taylor approximation: 
\begin{align}\label{eqn:cosine-Taylor}
    | \cos x | \geq 1 - \frac{x^2}{2}, \quad \text{for $x \in \bR$}.
\end{align}
Choosing arbitrary integers $p$ and $Q(p)$ satisfying the above inequality and combining the above inequalities, we have:
\begin{align}
    |\cos p| = |\cos (p-\pi Q(p))| 
    &\geq 1 - \frac{(p - \pi Q(p))^2}{2} 
    > 1 - \frac{1}{2Q(p)^{2(\mu(\pi)-1-\eps)}} \label{eqn:proof-der-1}
\end{align}
Raising \eqref{eqn:proof-der-1} to $p^\gamma$ and applying 
%yields:
%\begin{equation}
%|\cos p|^{p^\gamma} \geq \left ( 1 - \frac{1}{2q^{2(\mu(\pi)-1-\eps)}} \right %)^{p^\gamma}
%\end{equation}
the Bernoulli inequality we have:
\begin{equation}\label{eqn:Bernoullified}
    |\cos p|^{p^\gamma} 
    \geq \left ( 1 - \frac{1}{2Q(p)^{2(\mu(\pi)-1-\eps)}} \right )^{p^\gamma} 
    > 1 - \frac{p^\gamma}{2Q(p)^{2(\mu(\pi)-1-\eps)}}
\end{equation}
%
%Taking the limit of \eqref{eqn:Bernoullified} gives:
%\begin{align}
%    \lim_{p \to \infty}|\cos p|^{p^\gamma} 
%    &\geq \lim_{p \to \infty} \left ( 1 - \frac{1}{2Q(p)^{2(\mu(\pi)-1-\eps)}} \right )^{p^\gamma} \\
%    &\geq  1 - \lim_{p \to \infty}\frac{p^\gamma}{2Q(p)^{2(\mu(\pi)-1-\eps)}} \notag
%\end{align}
%
Observe $\limsup |\cos k|^{k^\gamma} \leq 1$. Let $\gamma < 2(\mu(\pi)-1 - \eps)$. We have $p^\gamma = o(Q(p)^{2(\mu(\pi) - 1 -\eps)})$ because $Q(p) = \mathcal{O}(p)$ by Lemma~\ref{thm:p/Q(p)->pi}. Since we can choose arbitrarily large $p$ and $Q(p)$ satisfying \eqref{eqn:approximability-inequality}, we can bring $|\cos p|^{p^\gamma}$ arbitrarily close to $1$ by \eqref{eqn:Bernoullified}. This establishes \eqref{eqn:main-result-g-case-i}. %{eqn:1_for_g<2mu}. \\ %Thus, when $\gamma < 2(\mu(\pi) - 1)$, $\limsup |\cos k|^{k^\gamma} = 1$.
\\

\underline{Step 2:} we show \eqref{eqn:main-result-g-case-ii}. 
% that 
% \begin{equation}\label{eqn:0_for_g>2mu}
%     \limsup \{|\cos k|^{k^\gamma}\}_{k \in \bN} = 0,
%     \qquad\textrm{for all } \gamma > 2\mu(\pi) - 2. 
% \end{equation}
First, observe that $|\cos x| \leq 1 - \frac{x^2}{2} + \frac{x^4}{24}$ for $x \in [-\frac{\pi}{2}, \frac{\pi}{2}]$. Here, the interval $[-\frac{\pi}{2}, \frac{\pi}{2}]$ is chosen arbitrarily; any interval which is small enough to ensure the inequality would suffice (the inequality does not hold for all $x$). Fix $\eps > 0$ and let $\gamma = 2(\mu(\pi) - 1 + \eps)$. Then for arbitrary integers $p$ and $Q(p)$, $|p - \pi Q(p)| \leq \frac{\pi}{2}$,  we have:
\begin{equation}\label{gammatozeroeq}
|\cos p|^{p^\gamma} \leq \left ( 1 - \frac{(p-\pi Q(p))^2}{2} + \frac{(p-\pi Q(p))^4}{24}\right )^{p^\gamma}
\end{equation}
By Definition~\ref{def:irrationality-measure}, for any real number $\lambda > 0$, there exists an integer $N$ such that %for all $p > p_0$ and $q > q_0$ there are no $p$ and $q$ satisfying the following relation:
\begin{equation}\label{reverse_ineq}
|p-q\pi| > \frac{1}{q^{\mu(\pi) - 1 + \lambda}},
\qquad \text{for all } p,q > N.
\end{equation}
Since $1 - \frac{x^2}{2} + \frac{x^4}{24}$ is monotonically increasing on $[-\frac{\pi}{2}, 0)$ and monotonically decreasing on $(0, \frac{\pi}{2}]$, for $p, Q(p) > N$, we have:
\begin{align}
    |\cos p|^{p^\gamma} 
    &\leq \left ( 1 - \frac{(p-\pi Q(p))^2}{2} + \frac{(p-\pi Q(p))^4}{24}\right )^{p^\gamma} \nonumber \\
    &\leq \left ( 1 - \frac{1}{2Q(p)^{2(\mu(\pi) - 1 + \lambda)}} + \frac{1}{24Q(p)^{4(\mu(\pi) - 1 + \lambda)}} \right )^{p^\gamma} \nonumber \\
    &= \left ( 1 - \frac{12Q(p)^{2(\mu(\pi) - 1 + \lambda)} - 1}{24Q(p)^{4(\mu(\pi) - 1 + \lambda)}} \right )^{p^\gamma}
    \label{eq:upper-bound}
\end{align}
Observe that if we define a function $\phi(p)$ such that:
\begin{align*}
    \frac{1}{\phi(p)} 
    = \frac{12Q(p)^{2(\mu(\pi) - 1 + \lambda)} - 1}{24Q(p)^{4(\mu(\pi) - 1 + \lambda)}},
\end{align*}
then $\phi(p) = \mathcal{O}(Q(p)^{2(\mu(\pi)-1+\lambda})$. If $\gamma > 2(\mu(\pi) - 1 + \lambda)$, then we have $\phi(Q(p)) = o(p^\gamma)$. Since we can choose arbitrarily large $p$ and $Q(p)$ satisfying \eqref{reverse_ineq}, by Lemma~\ref{thm:phi-identity} we can bring \eqref{gammatozeroeq} arbitrarily close to $0$.
%$$
%0 \leq \lim_{p \to \infty}|\cos p|^{p^\gamma} \leq \lim_{p \to \infty} \left ( 1 - \frac{12Q(p)^{2(\mu(\pi) - 1 + \lambda)} - %1}{24Q(p)^{4(\mu(\pi) - 1 + \lambda)}} \right )^{p^\gamma} = 0
%$$

Thus $\limsup \{|\cos k|^{k^\gamma}\}_{k \in \bN} = 0$ for all $\gamma > 2(\mu(\pi) - 1 + \lambda)$.
%, and $\limsup \{|\cos k|^{k^\gamma}\}_{k \in \bN} = 1$ for all $\gamma < 2(\mu(\pi) - 1)$. 
Since our choice of $\lambda$ is arbitrary but greater than zero, this establishes \eqref{eqn:main-result-g-case-ii}. %{eqn:0_for_g>2mu}. %we have $\limsup \Phi = 2(\mu(\pi) - 1)$, which is exactly what we wanted to prove.
By  \eqref{eqn:main-result-g-case-i}--\eqref{eqn:main-result-g-case-ii} %{eqn:1_for_g<2mu} and  \eqref{eqn:0_for_g>2mu}, 
the proof is complete.
\end{proof}

\section{Further investigations}

%should this title be changed, since we haven't assumed mu(pi) for any of the maths below?
%\subsection{Exploring the case where $\mu(\pi) = 2$}
%We will assume that $\mu(\pi) = 2$ and explore the behavior of the sequence $a_n =  |\cos n|^{n^2}$. 

%\subsubsection{Exact limsup value}

\begin{theorem}
   %If $\mu(\pi) = 2$, then $\liminf |\cos n|^{n^2} \geq \exp(-\pi^2/2) \approx 0.007192$.
   %shouldn't it be limsup? Because liminf is always zero, right?
   $\limsup |\cos n|^{n^2} \geq \exp(-\pi^2/2) \approx 0.007192$.
\end{theorem}
\begin{proof}
%If $\mu(\pi) = 2$, then we know that $\inf\{\mu :|U(\pi, \mu)| < \infty\} = 2$, and furthermore by diophantine approximation, we know %that $2 \notin \{\mu :|U(\pi, \mu)| < \infty\}$. With this in mind, \eqref{eqn:proof-der-1} becomes:
By Diophantine approximation and \eqref{eqn:cosine-Taylor}, \eqref{eqn:proof-der-1} becomes:
\begin{align}
    |\cos p| = |\cos (p-\pi Q(p))| 
    &\geq 1 - \frac{(p - \pi Q(p))^2}{2} 
    > 1 - \frac{1}{2Q(p)^2} \label{eqn:proof-der-1-mu2}
\end{align}

Raising \eqref{eqn:proof-der-1-mu2} to the $p^2$, we obtain
\begin{align}
    |\cos p|^{p^2} \geq \left ( 1 - \frac{1}{2Q(p)^2} \right )^{p^2}.
    \label{eqn:proof-der-2-mu2}
\end{align}
Fix $\eps > 0$. Since $\frac{p}{Q(p)} \to \pi$ by Lemma~\ref{thm:p/Q(p)->pi}, we have $\pi^2-\eps < \frac{p^2}{Q(p)^2} < \pi^2 + \eps$ for all sufficiently large $p$, whence
\begin{align*}
    \left(1-\frac{\frac{\pi^2-\eps}{2}}{p^2}\right)^{p^2}
    > \left(1-\frac{p^2}{2p^2Q(p)^2}\right)^{p^2}
    > \left(1-\frac{\frac{\pi^2+\eps}{2}}{p^2}\right)^{p^2}.
\end{align*}
It follows that the limit of the right side of \eqref{eqn:proof-der-2-mu2} is $\exp(-\pi^2/2)$.
\end{proof}
% Because $p$ and $Q(p)$ grow at the same rate, as $p$ becomes arbitrarily large we have:
% \begin{align}
%     |\cos p|^{p^2} \geq \left ( 1 - \frac{p^2}{2p^2Q(p)^2} \right )^{p^2} \to \left ( 1 - \frac{\frac{\pi^2}{2}}{p^2} \right )^{p^2} \to \exp\left (-\frac{\pi^2}{2} \right ) \label{eq:mu2-lower-bound}
% \end{align}

% We follow the same process with \eqref{gammatozeroeq} to obtain:
% \begin{align}
%     |\cos p|^{p^2} \leq \left ( 1 - \frac{12Q(p)^{2 + \lambda} - 1}{24Q(p)^{4 + \lambda}} \right )^{p^\gamma}  \to \left ( 1 - \frac{\frac{\pi^2}{2}}{p^2} \right )^{p^2} \to \exp\left (-\frac{\pi^2}{2} \right ) \label{eq:mu2-upper-bound}
% \end{align}

% From \eqref{eq:mu2-lower-bound} and \eqref{eq:mu2-upper-bound}, we see that $\limsup |\cos n|^{n^2} = \exp(-\pi^2/2) \approx 0.007192$.

\subsection{Cosine Identities}
We manipulate the following identity to obtain insight on the sequence:
%From Wikipedia:
\begin{align} 
\prod_{k=1}^n \cos \theta_k & = \frac{1}{2^n}\sum_{e\in S} \cos(e_1\theta_1+\cdots+e_n\theta_n), 
\label{eqn:cosine-identity}
\end{align}
where $S=\{1,-1\}^n$. %If $\theta_k=n$ for all $k$ and we let $i$ be the number of times $1$ %appears in $e$. 

\begin{theorem}\label{cosine-identity-estimate}
We have the estimate
\begin{align}
    \cos(n)^{n^2} 
    \geq %1 - \frac{n^2}{2n^{2\mu(\pi)-2}} = 
    1 - \frac{n^{4-2\mu(\pi)}}{2}. \label{eq:quad_sum} %There was a ^2 missing here
\end{align}
\end{theorem}
\begin{proof}
    Choosing $\theta_k=n$ and $S=\{1,-1\}^{n^2}$ in \eqref{eqn:cosine-identity}, the identity becomes:
    \begin{align*} 
     \cos (n)^n & = \frac{1}{2^n}\sum_{e\in S} \cos((e_1+\cdots+e_n)n) 
    %\end{align*}
    %After elementary combinatorics, we have:
    %\begin{align*} 
    %\cos(n)^n 
    = \frac{1}{2^n}\sum_{i = 0}^{n} \binom{n}{i}\cos\left((2i-n)n\right).
    \end{align*}
    Observe that for $\cos(n)^{n^2}$, we have:
    \begin{equation}
    \cos(n)^{n^2} = \frac{1}{2^{n^2}}\sum_{i = 0}^{{n^2}} \binom{n^2}{i}\cos\left((2i-n^2)n\right) \label{eq:cos_sum}
    \end{equation}
    Applying the cosine approximation \eqref{eqn:cosine-Taylor} and writing $Q(n) = \left[\frac n\pi\right]$, for large choices of $n$ which satisfy \eqref{eqn:approximability-inequality}, we have:
    \begin{align} 
        \cos(n)^{n^2} 
        & \geq \frac{1}{2^{n^2}}\sum_{i = 0}^{{n^2}} \binom{n^2}{i}\left(1- \frac{(2i-n^2)^2(n - \pi \left[\frac n\pi\right])^2}{2}\right) \notag \\
        &= \frac{1}{2^{n^2}}\sum_{i = 0}^{{n^2}} \left( \binom{n^2}{i} - \binom{n^2}{i}\frac{(2i-n^2)^2}{2}(n - \pi \left[\tfrac n\pi\right])^2\right) \notag \\
        %&=1-\frac{n^2 (1 - \tfrac\pi n \left[\tfrac n\pi\right])^2}{2\cdot2^{n^2}}\sum_{i = 0}^{{n^2}}\binom{n^2}{i}(2i-n^2)^2 \notag \\
        &=1-\frac{(n - \pi Q(n))^2}{2\cdot2^{n^2}}\sum_{i = 0}^{{n^2}}\binom{n^2}{i}(2i-n^2)^2 \notag \\
        &= 1- \frac{n^2(n-\pi Q(n))^2}{2} 
    \end{align}
    where we have used the identity $\sum_{i = 0}^{{n^2}}\binom{n^2}{i}(2i-n^2)^2 = n^2 2^{n^2}$ in the last step. Definition~\ref{def:irrationality-measure} then yields 
    \[\cos(n)^{n^2} 
    \geq 1 - \frac{n^2}{2n^{2\mu(\pi)-2}}\]
    and \eqref{eq:quad_sum} follows. %provides for the further estimate
\end{proof}

From \eqref{eq:quad_sum}, we can see that the existence of a lower bound (as we pick larger $n$ which better approximate multiples of $\pi$) is determined by the growth rate of the second term. If the rate is $\mathcal{O}(1)$, then there is a positive lower bound. If the growth rate is $o(1)$, then the lower bound is 1, and the limsup is also 1.

Additionally, we can get a more precise lower bound of \eqref{eq:quad_sum} by adding more terms to the Taylor expansion (note that we continue choose large values of $n$ which satisfy \eqref{eqn:approximability-inequality}):
\begin{align}
    \cos(n)^{n^2}
    &= 1 - \frac{n^2(n-\pi Q(n))^2}{2!} + \frac{n^2(3n^2-2)(n-\pi Q(n))^4}{4!} \notag \\
    &\rule{20mm}{0mm} - \frac{n^2(16 + 15 n^2 (n^2 - 2))(n-\pi Q(n))^6}{6!}... %\notag \\
    %&\geq 1 - \frac{n^2}{2!n^{2\mu(\pi) - 2}} + \frac{n^2(3n^2-2)}{4!n^{4\mu(\pi) - 4}} - \frac{n^2(16 + 15 n^2 (n^2 - 2))}{6!n^{6\mu(\pi) - 6}}... \label{eq:full_expansion}
\end{align}
where all of the identities of the form $\sum_{i=0}^{n^2}\binom{n^2}{i}(2i-n^2)^a$ were computed with Mathematica. 
Thus, we have shown the following.
\begin{corollary}
The bound in Theorem~\ref{cosine-identity-estimate} can be improved to
\begin{align}
    \cos(n)^{n^2}
    &\geq 1 - \frac{n^2}{2!n^{2\mu(\pi) - 2}} + \frac{n^2(3n^2-2)}{4!n^{4\mu(\pi) - 4}} - \frac{n^2(16 + 15 n^2 (n^2 - 2))}{6!n^{6\mu(\pi) - 6}}... \label{eq:full_expansion}
\end{align}
\end{corollary}
Now it is clear from \eqref{eq:full_expansion} that if we assume $\mu(\pi) = 2$, then only the highest coefficient of the numerator will remain when we take the $\limsup$:
\begin{align}
    \limsup |\cos(n)^{n^2}|
    &\geq 1 - \frac{1}{2!} + \frac{3}{4!} - \frac{15}{6!} + \frac{105}{8!} - \frac{945}{10!}... \approx 0.6065 \label{eq:full_limsup}
\end{align}
Note that the numerator coefficients we use in \eqref{eq:full_limsup} were from identities found using Mathematica. A full table of these coefficients up to denominator $16!$ is shown below:

%Mathematica code:
%Table[Coefficient[FullSimplify[Sum[Binomial[n^2, i]*(2 i - n^2)^(2 a), {i, 0, n^2}]]/2^(n^2), n^(2 a)], {a, 8}]

\begin{center}
\begin{tabular}{ |c|c| } 
 \hline
 Denominator & Numerator \\
 \hline
 $2!$ & $-1$ \\ 
 $4!$ & $3$ \\
 $6!$ & $-15$ \\ 
 $8!$ & $105$ \\ 
 $10!$ & $-945$ \\ 
 $12!$ & $10395$ \\ 
 $14!$ & $-135135$ \\
 $16!$ & $2027025$ \\ 
 \hline
\end{tabular}
\end{center}

Since we are determining these coefficients through Mathematica, we do not know the growth rate of the coefficients. It is possible that the series could diverge, but as long as we truncate the series at a negative term, the lower bound holds.

\begin{corollary}
   %If $\mu(\pi) = 2$, then $\liminf |\cos n|^{n^2} \geq \exp(-\pi^2/2) \approx 0.007192$.
   %shouldn't it be limsup? Because liminf is always zero, right?
   $\limsup |\cos n|^{n^2} \geq 0.6065...$.
\end{corollary}
\begin{proof}
$\mu(\pi) \geq 2$, by transcendentality of $\pi$. If $\mu(\pi) = 2$, then $\limsup |\cos n|^{n^2} \geq 0.6065...$ by \eqref{eq:full_limsup}. If $\mu(\pi) > 2$, then $\limsup |\cos n|^{n^2} = 1$ by Theorem \ref{thm:main-result}. This completes the proof.
\end{proof}

\begin{figure}
    \addheight{\includegraphics[width=60mm]{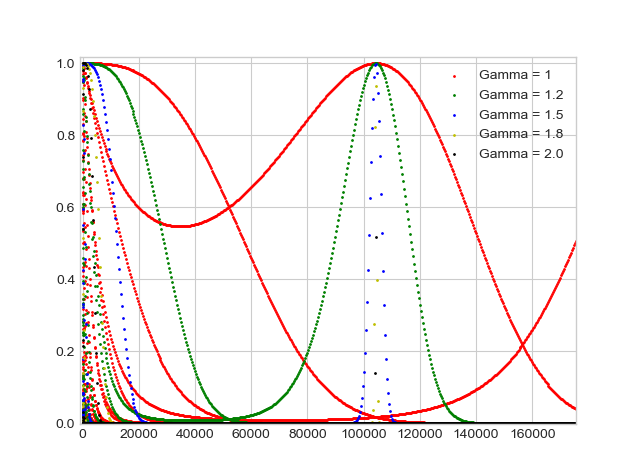}
    \qquad
    \includegraphics[width=60mm]{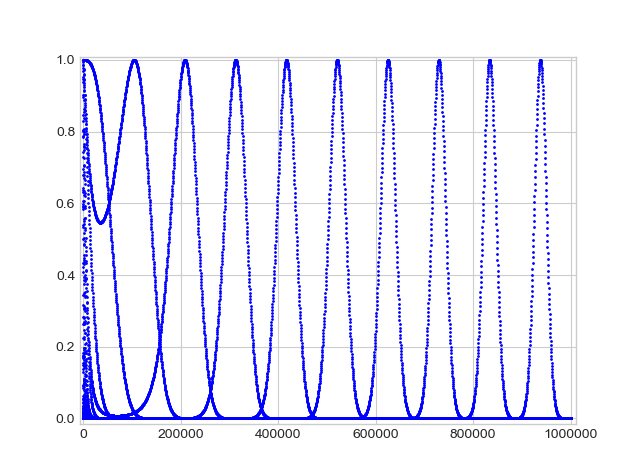}}
      
    \caption{\small Varying values of $\gamma$, plotted from 0 to 170000.   $\gamma = 1$ plotted from 0 to $10^6$.}
    \label{fig:plots_row1}
\end{figure}
\begin{figure}
    \addheight{\includegraphics[width=60mm]{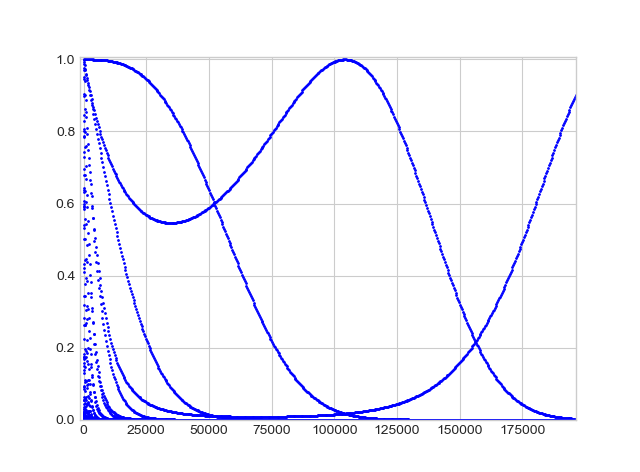} 
    \qquad
    \includegraphics[width=60mm]{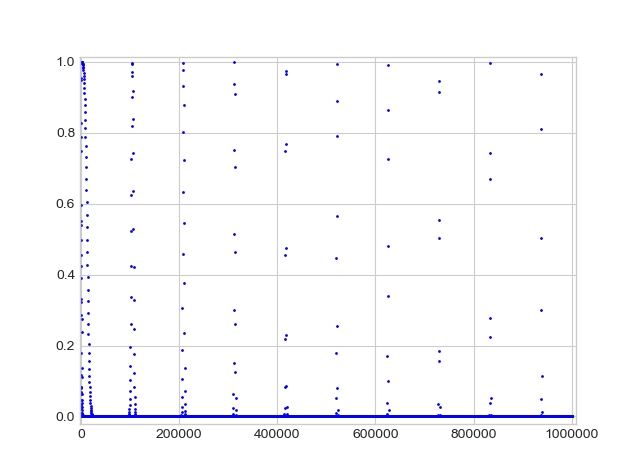}} \\
    \caption{\small $\gamma = 1$ plotted from 0 to 200000.  $\gamma = 1.5$ plotted from 0 to $10^6$.}
    \label{fig:plots_row2}
\end{figure}

\section{Numerical results}
\label{sec:numerics}
Computationally, we find values of $|\cos k|^{k^\gamma}$ very close to one for values of $k$ on the order of $10^8$, with values of $\gamma$ very close to 2. When we set $\gamma = 2$, we do not find values of $k$ close to one. This is evidence that $\mu(\pi) = 2$, but a proof has evaded us so far. Various values of $\gamma$ are plotted in Figure~\ref{fig:plots_row1} and Figure~\ref{fig:plots_row2}.

\noindent
We see that for larger values of $\gamma$, points near one get more sparse but still reach values very close to 1. In the case of $\gamma = 1$, there are interesting subsequences, some of which persist quasiperiodically or die out entirely. Investigations of these subsequences are an interesting subtopic which is open to exploration. Figure~\ref{fig:plots_4up} shows plots of $\gamma$ near and above 2.

\begin{figure}
\noindent
\begin{tabular}{cc}
      \addheight{\includegraphics[width=60mm]{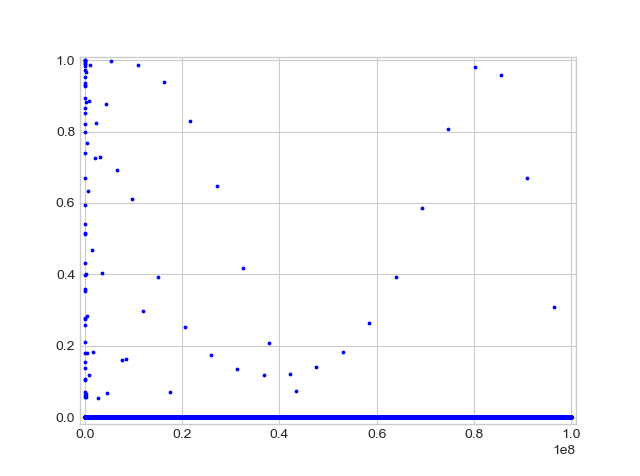}} &
      \addheight{\includegraphics[width=60mm]{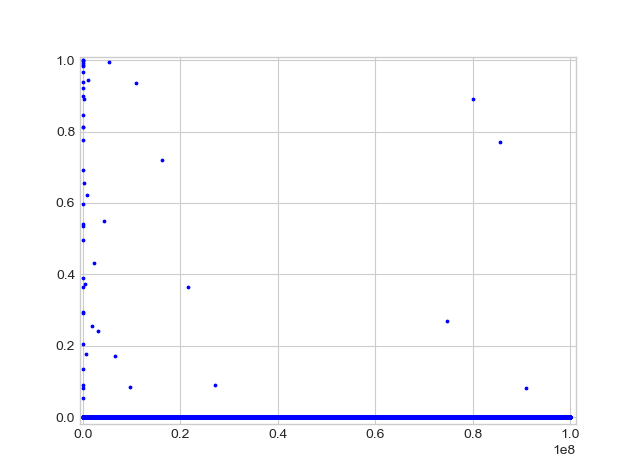}} \\
      \small $\gamma = 1.8$ & $\gamma = 1.9$ \\
\end{tabular}\\
\noindent
\begin{tabular}{cc}
      \addheight{\includegraphics[width=60mm]{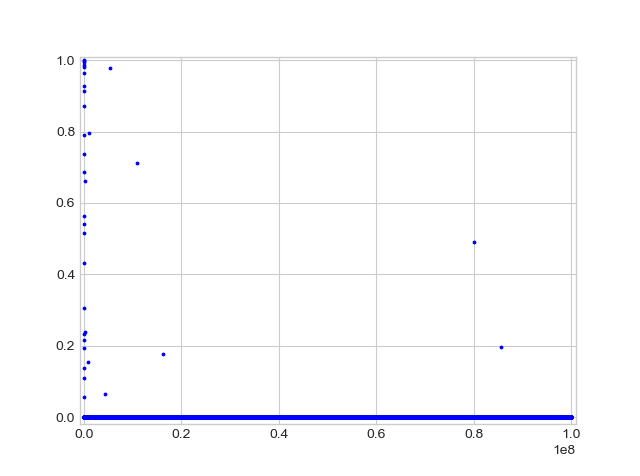}} &
      \addheight{\includegraphics[width=60mm]{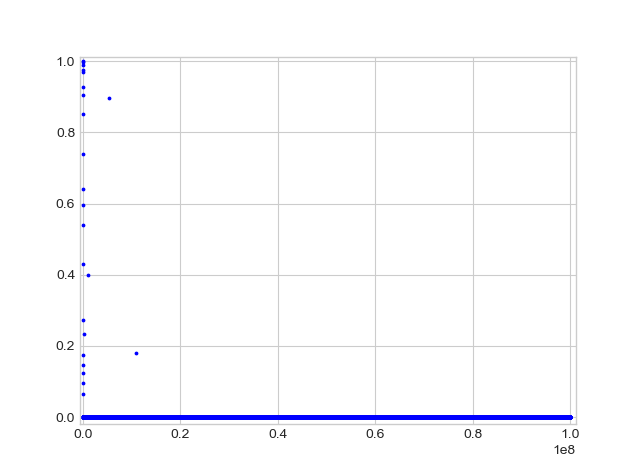}} \\
      \small $\gamma = 2.0$ &  $\gamma = 2.1$ \\
\end{tabular}\\
 \caption{\small Plots of $\left ( \cos n \right )^{n^\gamma}$, for $\gamma \approx 2$.}
    \label{fig:plots_4up}
\end{figure}

When $\gamma$ is just below two, we still see subsequences that are close to $1$, but they become more sparse as $\gamma$ is closer to $2$. In the case where $\gamma = 2.1$, we see that the values approach zero, and do not appear to take on values close to one in the range sampled, which is expected if $\mu(\pi) = 2$.

\section{Subsequences of \protect{$( \cos n )^{n^\gamma}$}}
\label{sec:subsequences}

Various interesting subsequences appear in $\left ( \cos n \right )^{n^\gamma}$, and the following theorems shed some light on the topic.

\begin{theorem}\label{thm:main-subsequence-theorem}
   For all $0 \leq \gamma < 2(\mu(\pi) - 1)$, for any real number $\alpha \in (0, 1)$, there exist infinitely many arbitrarily long arithmetic progressions $\mathbf{S}$ such that for all $a_n \in \mathbf{S}$, $\left |\cos a_n \right|^{a_n^\gamma} > \alpha$.
\end{theorem}
\begin{proof}
The proof relates closely to the proof of \eqref{eqn:main-result-g-case-i}. We modify \eqref{eqn:Bernoullified} to obtain:
$$
|\cos (np)|^{(np)^\gamma} \geq \left ( 1 - \frac{(n(p-\pi Q(p))^2}{2}\right )^{(np)^\gamma} > 1 - \frac{n^{\gamma + 2}p^\gamma}{2Q(p)^{2(\mu(\pi)-1-\epsilon)}}
$$
By \eqref{eqn:approximability-inequality}, for any $n \in \mathbf{N}$, we can bring $|\cos (np)|^{(np)^\gamma}$ arbitarily close to $1$, which is greater than $\alpha$. Also observe that for any $a \in \mathbf{N},\ a < n$ we have: %does this count as induction??
$$
\frac{n^{\gamma + 2}p^\gamma}{2Q(p)^{2(\mu(\pi)-1-\epsilon)}} > \frac{a^{\gamma + 2}p^\gamma}{2Q(p)^{2(\mu(\pi)-1-\epsilon)}}
$$
Thus for any series of length $n \in \mathbf{N}$, there exists infinitely many integers $p$ such that the sequence $\{|\cos (kp)|^{(kp)^\gamma},\ k=0,1,2,...,n \}$ is bounded below by $\alpha \in (0, 1)$. This completes the proof.
\end{proof}

\begin{definition}\label{def:persistent-subsequences}
    We refer to the subsequences which Theorem~\ref{thm:main-subsequence-theorem} guarantees to exist as \emph{persistent subsequences}.
\end{definition}

The terminology of Definition~\ref{def:persistent-subsequences} stems from the idea that these values of $\cos(n_k)$ are sufficiently close to 1 that not even raising them to the (typically extremely large) number $n_k$ will result in a number near 0.

Numerical experimentation reveals that \textbf{each peak} in Figure~\ref{fig:1000000} corresponds to a separate persistent subsequence 
\begin{align}
    \cos(n_{k})^{n_{k}}, \qquad \text{where } n_{k} = 355k + (3+22j), \quad k\in \bN,
\end{align}
for some fixed $j \in \bN$.
The appearance of $p=355$ is due to its role in the continued fraction expansion of $\pi$ and the fact that $\frac{355}{113}$ is the best rational approximation of $\pi$ with denominator less than $16604$, and the appearance of $22$ is also likely due to $\pi \approx \frac{22}{7}$. This suggests that other persistent subsequences may be found by looking at other exceptionally accurate fractional approximations to $\pi$, for example: $\frac{52163}{16604}$, $\frac{833719}{265381}$ and $\frac{42208400}{13435351}$. Indeed, we find persistent subsequences for $n_k = 833719k$ and $n_k = 833719k+42208400$, as in Figure~\ref{fig:far-out-subsequences}.

\begin{figure}
    \addheight{\includegraphics[width=50mm]{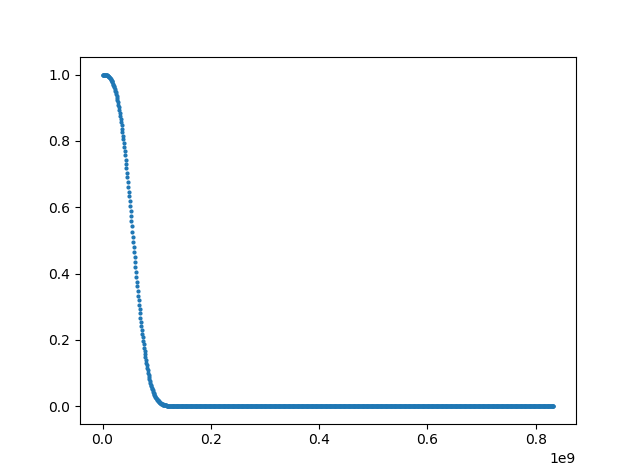}\includegraphics[width=50mm]{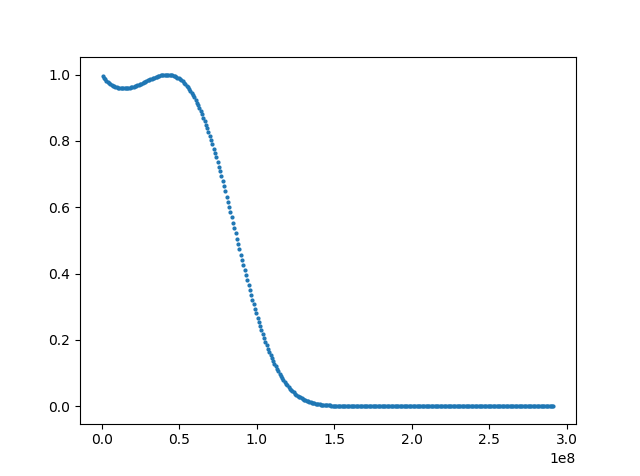}}
    \caption{\small Left: the persistent subsequence for $n_k = 833719k$. Right: the persistent subsequence for $n_k = 833719k+42208400$. Notice: the scale on the horizontal axis is $\times 10^8$.}
    \label{fig:far-out-subsequences}
\end{figure}

Each persistent subsequence appears to have a single peak, of a shape very similar to a Gaussian distribution (for those whose peak is not too close to 0). Curve matching in MATLAB reveals that these peaks match well. For example, the persistent subsequence $\cos(n_k)^{n_k}$ with $n_k = 644+355k$ has a coefficient of determination $R^2 = 0.9983$ with the Gaussian $0.9978 \ee^{-\left(\frac{x-260.8}{84.57}\right)^2}$. 
However, these curves are not truly Gaussian; the following theorem derives the general peak shape we observe in persistent subsequences and is matched against a sample subsequence in Figure~\ref{fig:contin-vs-discrete}.

\begin{theorem}\label{thm:smooth-form-theorem}
   For any arithmetic progression $a_n = pn + d$, the sequence $\{|\cos a_n|^{a_n^\gamma}\}$ lies on the curve 
   %$f(x) = |\cos(ax + b)|^{(cx+d)^\gamma}$, where $a, b \in \mathbf{R}$, $c, d \in \mathbf{Z}$.
    \begin{equation}\label{eq:curve-fit}
        f(x) = |\cos((p - \pi Q(p))\alpha(x) + d - \pi Q(d))|^{x^\gamma}.
    \end{equation}
\end{theorem}
\begin{proof}
The proof begins similarly to the proof of Theorem \ref{thm:main-subsequence-theorem}. Let $\mathbf{S}$ be an arithmetic progression of the form $S_n = pn + d$, where $p, d \in \mathbf{Z}$. Then we have:
   $$
   |\cos (pn + d)|^{(pn + d)^\gamma} = |\cos (n(p - \pi Q(p)) + (d - \pi Q(d)))|^{(pn + d)^\gamma}
   $$
   %$$
   %\geq \left ( 1 - \frac{(n(p - \pi Q(p)) + (d - \pi Q(d)))^2}{2} \right )^{(pn + d)^\gamma}
   %$$
   %\begin{equation}\label{eq:quadratic-form}
   %= \left (- \frac{(p-\pi Q(p))^2}{2}n^2 - (p-\pi Q(p))(d - \pi Q(d))n - \left ( \frac{(d - \pi Q(d))^2}{2} - 1 \right ) %\right)^{(pn + d)^\gamma}
   %\end{equation}
   %Notice that the closer that $pn + d$ is to zero, the more accurate the approximation becomes.
The domain can be extended into the reals via the transformation $\alpha(x) = \frac{x-d}{p}$, and this yields \eqref{eq:curve-fit}.
\end{proof}

\noindent
\begin{figure}
    \addheight{\includegraphics[width=100mm]{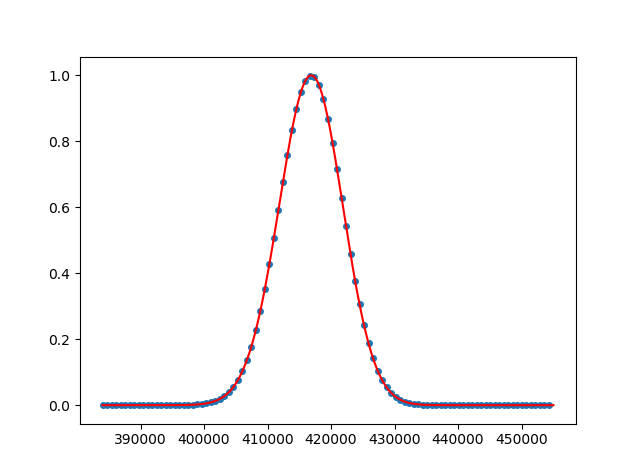}}
    \caption{\small $|\cos(8.49137\cdot10^{-8} x-0.0353982)|^{x^{1.2}}$ (red) plotted against the sequence $|\cos(a_n)|^{a_n^{1.2}}$ (blue).}
    \label{fig:contin-vs-discrete}
\end{figure}

% Some intriguing observations can be made.
% \begin{itemize}
%     \item The minimizer of $|1-\frac \pi n \left[\frac n\pi\right]|$ is $n=355$, for any $n < 52163$. However, above this regime, new minimizers appear frequently.
%     \item Certain subsequences interweave and cannot be disguished visually. Moreover, these groups appear with offsets differing by a multiple of 355. For example, see Figure~\ref{fig:groups}.
% \end{itemize}

% \begin{figure}
%     \addheight{\includegraphics[width=40mm]{groups1,pdf}}
%     \caption{\small $|\cos(8.49137\cdot10^{-8} x-0.0353982)|^{x^{1.2}}$ (red) plotted against the sequence $|\cos(a_n)|^{a_n^{1.2}}$ (blue).}
%     \label{fig:groups}
% \end{figure}

% \subsection{Computational Exploration of Theorem \ref{thm:smooth-form-theorem}}
% \label{sec:computations}

\appendix
\section{Irrationality measure of rational and algebraic numbers}

\begin{theorem}[{\cite[Thm.~E.2]{Bugeaud}}]
   If $\alpha \in \bQ$, then $\mu(\alpha)=1$. 
\end{theorem}
\begin{proof}
    Since $|q\alpha-[q\alpha]| < 1$ for every $q \in \bN$, we have $\mu(\alpha) \geq 1$. If $\alpha = \frac ab$ in lowest terms, and $\alpha \neq \frac pq$, then $|\alpha - \frac pq| \geq \frac1{|bq|}$, whence $\mu(\alpha) = 1$.
\end{proof}

Using continued fraction expansions, one can prove that $\mu(\alpha) \geq 2$ whenever $\alpha \in \bR \setminus \bQ$, but that is beyond the scope of this paper. However, we can show that all Liouville numbers are transcendental. It suffices to prove the following result; our treatment is adapted from \cite{Simmons}.

\begin{theorem}
   For all irrational algebraic numbers $\alpha$ of degree $n$, there exists a constant $c > 0$ such that: $\left | \alpha - \frac{p}{q} \right | \geq \frac{c}{q^n}$ for any $p,q \in \bZ$ with $q > 0$.
\end{theorem}

\begin{proof}
Let $P(x) = \sum^{n}_{k=0} a_k x^k$ be a polynomial of degree $n$ such that $P(\alpha) = 0$, and denote the set of rational roots of $P$ by $R = \{\zeta \in \bQ : P(\zeta) = 0\}$. For a given rational $r = \frac{p}{q}$, we consider the following three cases: (i) $|\alpha-r| \geq 1$, (ii) $r \in R$, and (iii) 
$r \notin R$ but $|\alpha-r| \leq 1$.

\textit{Case (i).} For any $r=\frac pq$ satisfying $|\alpha-r| \geq 1$, we have $|\alpha-r| \geq \frac{1}{q^n}$, and $c=1$ would work.

\textit{Case (ii).} For $r \in R$, we can define $\xi > 0$ by
\[\xi = \begin{cases}
    \min \{|\alpha-r| : r \in R\},& R \neq \varnothing \\
    1, & R = \varnothing,
    \end{cases}\]
and we immediately have $|\alpha-r| \geq \xi \geq \frac{\xi}{q^n}$.

\textit{Case (iii).} For $r \notin R$, we know that $P(r)$ is some multiple of $\frac{1}{q^n}$. Since $r \notin R$, this implies 
\begin{align}\label{eqn:P-dif-bound}
    |P(\alpha)-P(r)| = |P(r)| \geq \frac{1}{q^n}.
\end{align}
Observe that  
\[\alpha^k-r^k = (\alpha-r)\sum^{k-1}_{i=0} r^i \alpha^{k-1-i}.\]
Solving this identity for $r^k$ and substituting it into the definition of $P$ leads to
\[P(\alpha)-P(r) = (\alpha-r)\sum^{n}_{k=1}a_k \sum^{k-1}_{i}r^i\alpha^{k-1-i}.\]
Recall that for this case we have also assumed $|\alpha-r| \leq 1$. For such $r$, we have $|r| \leq |\alpha| + 1$, and so for $c_\alpha$ defined by
\[c_\alpha = \sum^{n}_{k=1}|a_k|k(|\alpha|+1)^{k-1},\]
we get $|P(\alpha)-P(r)| \leq |\alpha-r|c_\alpha$ by the triangle inequality.
In combination with \eqref{eqn:P-dif-bound}, this gives
\[|\alpha-r| 
    \geq \frac{|P(\alpha)-P(r)|}{c_\alpha} 
    \geq \frac{1}{c_\alpha q^n}.
\]
For $c = \min \left (1, \xi, \frac{1}{c_\alpha} \right )$, all cases are covered.
\end{proof}

\section*{Acknowledgements}
The first author would like to thank the second author for suggesting this problem and assisting greatly with the development of the paper and for the second author's continued mentoring and guidance. The first author would also like to thank Math StackExchange user ``T. Bongers'' \cite{TBongers} for his/her/their insight on a more specific version of the problem. The second author acknowledges that the first author did all the hard work, and is grateful to Jordan Rowley for raising the question about $\sin(n)^n$ that led to the current work ($\cos(n)^n$ turned out to be more convenient to study). 
%Dr. Erin Pearse Ph. D. for introducing me to this problem and continuing to invest his time guiding me through the tough process of reviewing and writing a formal academic paper.
Both authors would like to thank the referee for helpful comments which we feel improved the paper significantly.

\end{document}